\documentclass[12pt,reqno]{article}

\usepackage[usenames]{color}
\usepackage[colorlinks=true,
linkcolor=webgreen, filecolor=webbrown,
citecolor=webgreen]{hyperref}

\definecolor{webgreen}{rgb}{0,.5,0}
\definecolor{webbrown}{rgb}{.6,0,0}

\usepackage{amssymb}
\usepackage{graphicx}
\usepackage{amscd}
\usepackage{lscape}
\usepackage{tikz}
\usepackage{tikz-cd}
\usepackage{pgfplots}

\usetikzlibrary{matrix}
\usetikzlibrary{fit,shapes}
\usetikzlibrary{positioning, calc}
\tikzset{circle node/.style = {circle,inner sep=1pt,draw, fill=white},
        X node/.style = {fill=white, inner sep=1pt},
        dot node/.style = {circle, draw, inner sep=5pt}
        }
\usepackage{tkz-fct}

\usepackage{amsthm}
\newtheorem{theorem}{Theorem}

\newtheorem{proposition}[theorem]{Proposition}

\theoremstyle{definition}

\newtheorem{example}[theorem]{Example}

\usepackage{float}

\usepackage{graphics,amsmath}
\usepackage{amsfonts}
\usepackage{latexsym}
\usepackage{epsf}

\newcommand{\seqnum}[1]{\href{http://oeis.org/#1}{\underline{#1}}}

\begin{document}

\begin{center}
\vskip 1cm{\LARGE\bf On the $f$-Matrices of Pascal-like Triangles Defined by Riordan Arrays} \vskip 1cm \large
Paul Barry\\
School of Science\\
Waterford Institute of Technology\\
Ireland\\
\href{mailto:pbarry@wit.ie}{\tt pbarry@wit.ie}
\end{center}
\vskip .2 in

\begin{abstract} We define and characterize the $f$-matrices associated to Pascal-like matrices that are defined by ordinary and exponential Riordan arrays. These generalize the face matrices of simplices and hypercubes. Their generating functions can be expressed simply in terms of continued fractions, which are shown to be transformations of the generating functions of the corresponding $\gamma$- and $h$-matrices.  \end{abstract}

\section{Introduction}
Three special classes of regular polytope \cite{Ziegler} exist in every dimension: the regular simplex, the hypercube, and the cross-polytope. For each of these, it is usual to construct a lower triangular matrix that enumerates, for each dimension, the number of faces of each lower dimension of the polytope in question.

For the regular simplex, we obtain the (infinite) matrix \seqnum{A135278} that begins
$$\left(
\begin{array}{ccccccc}
 1 & 0 & 0 & 0 & 0 & 0 & 0 \\
 2 & 1 & 0 & 0 & 0 & 0 & 0 \\
 3 & 3 & 1 & 0 & 0 & 0 & 0 \\
 4 & 6 & 4 & 1 & 0 & 0 & 0 \\
 5 & 10 & 10 & 5 & 1 & 0 & 0 \\
 6 & 15 & 20 & 15 & 6 & 1 & 0 \\
 7 & 21 & 35 & 35 & 21 & 7 & 1 \\
\end{array}
\right),$$
or alternatively its reversal \seqnum{A074909}
$$\left(
\begin{array}{ccccccc}
 1 & 0 & 0 & 0 & 0 & 0 & 0 \\
 1 & 2 & 0 & 0 & 0 & 0 & 0 \\
 1 & 3 & 3 & 0 & 0 & 0 & 0 \\
 1 & 4 & 6 & 4 & 0 & 0 & 0 \\
 1 & 5 & 10 & 10 & 5 & 0 & 0 \\
 1 & 6 & 15 & 20 & 15 & 6 & 0 \\
 1 & 7 & 21 & 35 & 35 & 21 & 7 \\
\end{array}
\right),$$
(depending on convention). 

The first matrix is the ordinary Riordan array $\left(\frac{1}{(1-x)^2}, \frac{x}{1-x}\right)$, while the second matrix is its reversal.

For the case of the hypercube, the corresponding matrix \seqnum{A038207} and its reversal \seqnum{A013609} begin
$$\left(
\begin{array}{ccccccc}
 1 & 0 & 0 & 0 & 0 & 0 & 0 \\
 2 & 1 & 0 & 0 & 0 & 0 & 0 \\
 4 & 4 & 1 & 0 & 0 & 0 & 0 \\
 8 & 12 & 6 & 1 & 0 & 0 & 0 \\
 16 & 32 & 24 & 8 & 1 & 0 & 0 \\
 32 & 80 & 80 & 40 & 10 & 1 & 0 \\
 64 & 192 & 240 & 160 & 60 & 12 & 1 \\
\end{array}
\right)$$
and
$$\left(
\begin{array}{ccccccc}
 1 & 0 & 0 & 0 & 0 & 0 & 0 \\
 1 & 2 & 0 & 0 & 0 & 0 & 0 \\
 1 & 4 & 4 & 0 & 0 & 0 & 0 \\
 1 & 6 & 12 & 8 & 0 & 0 & 0 \\
 1 & 8 & 24 & 32 & 16 & 0 & 0 \\
 1 & 10 & 40 & 80 & 80 & 32 & 0 \\
 1 & 12 & 60 & 160 & 240 & 192 & 64 \\
\end{array}
\right).$$

These are the ordinary Riordan array $\left(\frac{1}{1-2x}, \frac{x}{1-2x}\right)$ and its reversal. We have that $$\left(\frac{1}{1-2x}, \frac{x}{1-2x}\right)=\mathbf{B}^2,$$ where $\mathbf{B}=\left(\frac{1}{1-x}, \frac{x}{1-x}\right)$  is the binomial matrix $\left(\binom{n}{k}\right)_{n,k \ge 0}$ (Pascal's triangle \seqnum{A007318}). For our purposes in this note, we can also regard the binomial matrix as the exponential Riordan array
$$\mathbf{B}=\left[e^x, x\right],$$ in which case the face matrix for the hypercubes is given by
$\mathbf{B}^2=\left[e^{2x}, x\right]$ (or its reversal).

We now note the following.

$$ \left(\frac{1}{(1-x)^2}, \frac{x}{1-x}\right)\cdot \mathbf{B}^{-1}=\left(\frac{1}{1-x},x\right),$$ is the matrix that begins
$$\left(
\begin{array}{ccccccc}
 1 & 0 & 0 & 0 & 0 & 0 & 0 \\
 1 & 1 & 0 & 0 & 0 & 0 & 0 \\
 1 & 1 & 1 & 0 & 0 & 0 & 0 \\
 1 & 1 & 1 & 1 & 0 & 0 & 0 \\
 1 & 1 & 1 & 1 & 1 & 0 & 0 \\
 1 & 1 & 1 & 1 & 1 & 1 & 0 \\
 1 & 1 & 1 & 1 & 1 & 1 & 1 \\
\end{array}
\right).$$
Similarly, we have
$$ \left[e^{2x}, x\right] \mathbf{B}^{-1}=\left[e^x,x\right]=\mathbf{B},$$ which begins
$$\left(
\begin{array}{ccccccc}
 1 & 0 & 0 & 0 & 0 & 0 & 0 \\
 1 & 1 & 0 & 0 & 0 & 0 & 0 \\
 1 & 2 & 1 & 0 & 0 & 0 & 0 \\
 1 & 3 & 3 & 1 & 0 & 0 & 0 \\
 1 & 4 & 6 & 4 & 1 & 0 & 0 \\
 1 & 5 & 10 & 10 & 5 & 1 & 0 \\
 1 & 6 & 15 & 20 & 15 & 6 & 1 \\
\end{array}
\right).$$
In both cases, we obtain centrally symmetric, or palindromic,  matrices, whose rows are the $h$-vectors of the polytopes in question. We shall call a lower-triangular matrix $(a_{n,k})_{n,k\ge 0}$ \emph{Pascal-like} if $a_{n,0}=a_{n,n}=1$ and $a_{n,n-k}=a_{n,k}$. If $M$ is a Pascal-like matrix, then the matrix given by the matrix product $M \cdot \mathbf{B}$, where $\mathbf{B}$ is the binomial matrix $(\binom{n}{k})$, will be callde the $f$-matrix (face matrix) of $M$. 

We can generalize the two Pascal-like matrices above using Riordan arrays in two ways. The first way is to use ordinary Riordan arrays, in which case we obtain the parameterized family given by \cite{Cons}
$$\left(\frac{1}{1-x},\frac{x(1+rx)}{1-x}\right),$$ where for instance $r=0$ corresponds to the binomial matrix $\mathbf{B}$. The second way is to use exponential Riordan arrays \cite{Exp}, where we obtain the parameterized family palindromic matrices given by
$$\left[e^x, x\left(1+\frac{rx}{2}\right)\right].$$

We have investigated the associated $\gamma$-matrices for these two families in a previous paper \cite{gamma}.

We shall now turn our attention to the associated $f$-matrices.

In the next section, we shall briefly cover some definitions and results that will provide the context of the rest of the paper.

\section{Relevant definitions and results}
An ordinary Riordan array \cite{Book, Survey, SGWW} is a lower-triangular invertible matrix whose elements $a_{n,k}$ are given by
$$a_{n,k}=[x^n] g(x)f(x)^k,$$ where $g(x)=1+g_1 x+ g_2 x^2+ \cdots$ and $f(x)=x+f_2 x^2+ f_3 x^3+\cdots$ are two power series, with coefficients drawn from a suitable ring. In our case this ring will be the ring of integers $\mathbb{Z}$. This array is denoted by $(g(x), f(x))=(g, f)$, where $x$ is a dummy variable, in the sense that
$$a_{n,k}=[x^n]g(x)f(x)^k=[u^n]g(u)f(u)^k.$$
The bivariate generating function of the array $(g,f)$ is given by
$$\frac{g(x)}{1-y f(x)}.$$
Such arrays form a group (the Riordan group), where the product is given by
$$(g(x), f(x))\cdot (u(x), v(x))= (g(x)u(f(x)), v(f(x)),$$ and we have
$$(g(x), f(x))^{-1}=\left(\frac{1}{g(\bar{f}(x))}, \bar{f}(x)\right),$$ where $\bar{f}(x)$ is the compositional inverse of $f(x)$. Thus $\bar{f}(x)$ is the solution $u$ to the equation $f(u)=x$ such that $u(0)=0$.

An exponential Riordan array \cite{Book, Riordan_Exp} is a lower-triangular invertible matrix whose elements $a_{n,k}$ are given by
$$a_{n,k}=\frac{n!}{k!}[x^n] g(x)f(x)^k,$$ where $g(x)=1+g_1 \frac{x}{1!}+ g_2 \frac{x^2}{2!}+ \cdots$ and $f(x)=\frac{x}{1!}+f_2 \frac{x^2}{2!}+ f_3 \frac{x^3}{3!}+\cdots$ are two (exponential) power series, with coefficients drawn from a suitable ring. This array is denoted by $[g(x), f(x)]=[g, f]$.
The product rule and the inverse of an exponential Riordan array are calculated in a similar fashion to the ordinary case.

The bivariate generating function of the matrix $[g(x), f(x)]$ is given by
$$g(x)e^{yf(x)}.$$

These two variants are specializations of the case of so-called ``generalized Riordan arrays'' \cite{Wang}, which are defined in terms of two power series $g(x)=1+g_1 \frac{x}{c_1}+ g_2 \frac{x^2}{c_2}+ \cdots$ and
$f(x)=\frac{x}{c_1}+f_2 \frac{x^2}{c_2}+ f_3 \frac{x^3}{c_3}+\cdots$ where $c_n$ is a suitable sequence of non-zero coefficients. In this case, we have
$$a_{n,k}= \frac{c_n}{c_k} [x^n] g(x)f(x)^k.$$
We denote this array by $[g(x), f(x)]_{c_n}$.

\begin{example} The triangle of Narayana numbers $N_{n,k}=\frac{1}{k+1}\binom{n}{k}\binom{n+1}{k}$ \seqnum{A001263} is the matrix of $h$-vectors for the associahedron. This matrix is given by the generalized Riordan array
$$\left[\frac{I_1(2 \sqrt{x})}{\sqrt{x}}, x\right]_{n!(n+1)!}.$$ (Observation by Peter Bala, \seqnum{A001263}).
\end{example}

A Jacobi continued fraction is a continued fraction \cite{Wall} of the form
$$\cfrac{1}{1-\alpha x -
\cfrac{\beta x^2}{1-\gamma x -
\cfrac{\delta x^2}{1-\cdots}}}.$$ We use the notation
$$\mathcal{J}(\alpha, \gamma, \ldots; \beta, \delta, \ldots)$$ for such a fraction. The $k$-th binomial transform of such a continued fraction is then given by \cite{CFT}
$$\mathcal{J}(\alpha+k, \gamma+k, \ldots; \beta, \delta, \ldots).$$ If $a_n$ is a sequence, then its $k$-th binomial transform is the sequence $b_n=\sum_{i=0}^n \binom{n}{i}k^{n-i} a_i$. 

Sequences in this note, where known, will be referenced by their $Annnnnn$ number from the On-Line Encyclopedia of Integer Sequences \cite{SL1, SL2}. All the lower-triangular matrices that we shall encounter are infinite in extent. We display a suitable truncation. 

\section{The $f$-matrix of $\left(\frac{1}{1-x}, \frac{x(1+rx)}{1-x}\right)$}
We have that the $f$-matrix of the Pascal-like array $\left(\frac{1}{1-x}, \frac{x(1+rx)}{1-x}\right)$ is given by
$$F_r=\left(\frac{1}{1-x}, \frac{x(1+rx)}{1-x}\right)\cdot \mathbf{B}=\left(\frac{1}{1-x}, \frac{x(1+rx)}{1-x}\right)\cdot \left(\frac{1}{1-x}, \frac{x}{1-x}\right).$$
This is equal to
$$F_r=\left(\frac{1}{1-2x-rx^2},\frac{x(1+rx)}{1-2x-rx^2}\right).$$
This matrix begins
$$\left(
\begin{array}{cccccc}
 1 & 0 & 0 & 0 & 0 & 0 \\
 2 & 1 & 0 & 0 & 0 & 0 \\
 r+4 & r+4 & 1 & 0 & 0 & 0 \\
 4 r+8 & 6 r+12 & 2 r+6 & 1 & 0 & 0 \\
 r^2+12 r+16 & 2 r^2+24 r+32 & r^2+15 r+24 & 3 r+8 & 1 & 0 \\
 6 r^2+32 r+32 & 15 r^2+80 r+80 & 12 r^2+72 r+80 & 3 r^2+28 r+40 & 4 r+10 & 1 \\
\end{array}
\right),$$ or in reversed form,
$$\left(
\begin{array}{cccccc}
 1 & 0 & 0 & 0 & 0 & 0 \\
 1 & 2 & 0 & 0 & 0 & 0 \\
 1 & r+4 & r+4 & 0 & 0 & 0 \\
 1 & 2 r+6 & 6 r+12 & 4 r+8 & 0 & 0 \\
 1 & 3 r+8 & r^2+15 r+24 & 2 r^2+24 r+32 & r^2+12 r+16 & 0 \\
 1 & 4 r+10 & 3 r^2+28 r+40 & 12 r^2+72 r+80 & 15 r^2+80 r+80 & 6 r^2+32 r+32 \\
\end{array}
\right).$$
For $r=0,1,2$ we get, respectively,
$$\left(
\begin{array}{cccccc}
 1 & 0 & 0 & 0 & 0 & 0 \\
 1 & 2 & 0 & 0 & 0 & 0 \\
 1 & 4 & 4 & 0 & 0 & 0 \\
 1 & 6 & 12 & 8 & 0 & 0 \\
 1 & 8 & 24 & 32 & 16 & 0 \\
 1 & 10 & 40 & 80 & 80 & 32 \\
\end{array}
\right), \left(
\begin{array}{cccccc}
 1 & 0 & 0 & 0 & 0 & 0 \\
 1 & 2 & 0 & 0 & 0 & 0 \\
 1 & 5 & 4 & 0 & 0 & 0 \\
 1 & 8 & 17 & 8 & 0 & 0 \\
 1 & 11 & 39 & 51 & 16 & 0 \\
 1 & 14 & 70 & 154 & 143 & 32 \\
\end{array}
\right),$$ and
$$\left(
\begin{array}{cccccc}
 1 & 0 & 0 & 0 & 0 & 0 \\
 1 & 2 & 0 & 0 & 0 & 0 \\
 1 & 6 & 4 & 0 & 0 & 0 \\
 1 & 10 & 22 & 8 & 0 & 0 \\
 1 & 14 & 56 & 72 & 16 & 0 \\
 1 & 18 & 106 & 248 & 220 & 32 \\
\end{array}
\right).$$ The case $r=0$ is that of the hypercube. 

Using the form of the bivariate generating function for a Riordan array, we have the following result. 
\begin{proposition} The bivariate generating function for the $f$-matrix of the Pascal-like array $\left(\frac{1}{1-x},\frac{x(1+rx)}{1-x}\right)$ is given by
$$\frac{1}{1-(y+2)x-r(y+1)x^2},$$  or in reversed form,
$$\frac{1}{1-(2y+1)x-ry(y+1)x^2}.$$
\end{proposition}
Using the reversed form, we obtain the sequence of generating functions
$$\frac{1}{1-x-ryx^2} \to \frac{1}{1-(y+1)x-ryx^2} \to \frac{1}{1-(2y+1)x-ry(y+1)x^2}$$ for, respectively, the $\gamma$-matrix \cite{gamma}, the $h$-matrix, and the $f$-matrix for the Pascal-like matrix $\left(\frac{1}{1-x}, \frac{x(1+rx)}{1-x}\right)$ (where this matrix is the $h$-matrix).

We have, for the matrix family $\left(\frac{1}{1-x}, \frac{x(1+rx)}{1-x}\right)$, 
$$ \gamma_{n,k}=\binom{n-k}{n-2k}r^k,$$ 
$$ h_{n,k}=\sum_{j=0}^k \binom{k}{j}\binom{n-j}{n-k-j}r^j,$$ and 
$$ f_{n,k}=\sum_{i=0}^n \sum_{j=0}^i \binom{i}{j}\binom{n-j}{n-i-j}r^j \binom{i}{k}.$$ 
This follows from previous work \cite{gamma, Cons} and the definition of the $f$-matrix.

\section{The $f$-matrix of the Pascal-like matrix $\left[e^x, x(1+rx/2)\right]$}
The $f$-matrix of the Pascal-like matrix $\left[e^x, x(1+rx/2)\right]$ is given by
$$\left[e^x, x(1+rx/2)\right]\cdot \mathbf{B}=\left[e^x, x(1+rx/2)\right]\cdot \left[e^x, x\right],$$ which is
$$eF_r=\left[e^x e^{x(1+rx/2)}, x(1+rx/2)\right]=\left[e^{2x+rx^2/2}, x(1+rx/2)\right].$$
The bivariate generating function for $eF_r$ is then given by
$$e^x e^{x(1+rx/2)} e^{yx(1+rx/2)}=e^{2x+rx^2/2} e^{yx(1+rx/2)}.$$ This matrix begins
$$\left(
\begin{array}{ccccc}
 1 & 0 & 0 & 0 & 0 \\
 2 & 1 & 0 & 0 & 0 \\
 r+4 & r+4 & 1 & 0 & 0 \\
 6r+8 & 9r+12 & 3r+6 & 1 & 0 \\
 3 r^2+24 r+16 & 2 \left(3 r^2+24 r+16\right) & 3 \left(r^2+10 r+8\right) & 6r+8 & 1 \\
\end{array}
\right)$$ or in reversed form 
$$\left(
\begin{array}{ccccc}
 1 & 0 & 0 & 0 & 0 \\
 1 & 2 & 0 & 0 & 0 \\
 1 & r+4 & r+4 & 0 & 0 \\
 1 & 2r+6 & 9 r+12 & 6 r+8 & 0 \\
 1 & 6 r+8 & 3(r^2+10r+8) & 2(3r^2+24r+16) & 3 r (r+8)+16 \\
\end{array}
\right).$$ 
For $r=0,1,2$ we obtain the triangles 
$$\left(
\begin{array}{ccccc}
 1 & 0 & 0 & 0 & 0 \\
 1 & 2 & 0 & 0 & 0 \\
 1 & 4 & 4 & 0 & 0 \\
 1 & 6 & 12 & 8 & 0 \\
 1 & 8 & 24 & 32 & 16 \\
\end{array}
\right), \left(
\begin{array}{ccccc}
 1 & 0 & 0 & 0 & 0 \\
 1 & 2 & 0 & 0 & 0 \\
 1 & 5 & 5 & 0 & 0 \\
 1 & 9 & 21 & 14 & 0 \\
 1 & 14 & 57 & 86 & 43 \\
\end{array}
\right),$$ and 
$$\left(
\begin{array}{ccccc}
 1 & 0 & 0 & 0 & 0 \\
 1 & 2 & 0 & 0 & 0 \\
 1 & 6 & 6 & 0 & 0 \\
 1 & 12 & 30 & 20 & 0 \\
 1 & 20 & 96 & 152 & 76 \\
\end{array}
\right).$$ The case $r=0$ corresponds to \seqnum{A013609}. 

The reversed form of $eF_r$ will then have bivariate generating function
$$e^{2xy+rx^2y^2/2}e^{x(1+rxy/2)}.$$
We have that
$$e^{2xy+rx^2y^2/2}e^{x(1+rxy/2)}=e^{(2y+1)x}e^{ry(y+1)x^2/2}.$$
Now the exponential generating function $e^{\frac{x^2}{2}}$ expands to give the sequence of aerated double factorials \seqnum{A001147}
$$1, 0, 1, 0, 3, 0, 15, 0, 105, 0, 945,\ldots$$ which has the ordinary generating function
$$\cfrac{1}{1-
\cfrac{x^2}{1-
\cfrac{2x^2}{1-
\cfrac{3x^2}{1-\cdots}}}}.$$ This leads to the following proposition \cite{CFT}.
\begin{proposition} The ordinary generating function of the reversal of $F_r$ is given by the continued fraction $$\cfrac{1}{1-(2y+1)x-
\cfrac{ry(y+1)x^2}{1-(2y+1)x-
\cfrac{2ry(y+1)x^2}{1-(2y+1)x-
\cfrac{3ry(y+1)x^2}{1-\cdots}}}}.$$
\end{proposition}
This is a Jacobi continued fraction, which we can write as
$$\mathcal{J}(2y+1,2y+1,2y+1,\ldots; ry(y+1), 2ry(y+1),3ry(y+1),\ldots).$$
We then have the following result.
\begin{proposition} The $\gamma$-matrix, the $h$-matrix, and the $f$-matrix of the Pascal-like matrix
$$\left[e^x, x(1+rx/2)\right]$$ have their ordinary generating functions given by, respectively,
$$\mathcal{J}(1,1,1,\ldots, ry, 2ry, 3ry,\ldots),$$
$$\mathcal{J}(y+1, y+1,y+1,\ldots; ry, 2ry, 3ry,\ldots),$$
and
$$\mathcal{J}(2y+1,2y+1,2y+1,\ldots; ry(y+1), 2ry(y+1),3ry(y+1),\ldots).$$
\end{proposition}
\begin{proof}
The $h$-matrix in question is the Pascal-like matrix $\left[e^x, x(1+rx/2)\right]$ itself. This has bivariate generating function $$e^x e^{yx(1+rx/2)}=e^{(y+1)x}e^{rx^2/2}.$$ This is the $(y+1)$-st binomial transform of the sequence with generating function $e^{rx^2/2}$, whence the assertion concerning the $h$-matrix. The statement regarding the $\gamma$-matrix is proven in \cite{gamma}.
\end{proof}

\section{Remarks on the associahedron and the permutahedron}

It can be shown that the $\gamma$-matrix, the $h$-matrix and the $f$-matrix for the associahedron of type A (which are \seqnum{A055151}, \seqnum{A001263} and \seqnum{A033282}, respectively) have the following ordinary generating functions:
$$\mathcal{J}(1,1,1,\ldots; y,y,y,\ldots),$$
$$\mathcal{J}(y+1,y+1,y+1,\ldots; y,y,y,\ldots),$$ and
$$\mathcal{J}(2y+1,2y+1,2y+1,\ldots; y(y+1), y(y+1),y(y+1),\ldots).$$

In like manner, we can show that the $\gamma$-matrix, the $h$-matrix and the $f$-matrix for the permutahedron (which are \seqnum{A101280}, \seqnum{A008292}, \seqnum{A019538}, respectively \cite{Fomin, Petersen}) have the following ordinary generating functions:
$$\mathcal{J}(1,2,3,\ldots; 2y,6y,12y,\ldots),$$
$$\mathcal{J}(y+1,2(y+1),3(y+1),\ldots; 2y,6y,12y,\ldots),$$ and
$$\mathcal{J}(2y+1,2(2y+1),3(2y+1),\ldots; 2y(y+1), 6y(y+1),12y(y+1),\ldots).$$

We see that the assignment 
$$\mathcal{J}(\alpha, \beta, \gamma,\ldots;a,b,c,\ldots) \mapsto \mathcal{J}(\alpha, 2\beta, 3\gamma,\ldots;2a,6b,12c,\ldots)$$ provides us with a transfer mechanism between the associahedron and related objects to the permutahedron and associated objects.

\section{Conclusion} In this note we have shown how the face-vectors matrix of the hypercube and the $n$-simplex can be generalized to a generalized ``f-matrix'' for Pascal-like matrices that are defined by ordinary and exponential generating functions, respectively. To each such Pascal-like matrix there is an associated $\gamma$-matrix and and $f$-matrix. The bivariate generating functions are related in a specific and simple pattern. This pattern carries over to the associahedron and the permutahedron, and indeed, to other polytopes. It would appear useful to consider generalized Riordan arrays \cite{Wang} as a context for these cases.

\bigskip
\hrule
\bigskip
\noindent 2010 {\it Mathematics Subject Classification}: Primary
11B83; Secondary 33C45, 42C-5, 15B36, 15B05, 14N10, 11C20.
\noindent \emph{Keywords:} Face vector, Pascal-like triangle, Riordan array, Narayana number, Eulerian number, associahedron, permutahedron.

\bigskip
\hrule
\bigskip
\noindent (Concerned with sequences
\seqnum{A001147},
\seqnum{A001263},
\seqnum{A007318},
\seqnum{A019538},
\seqnum{A033282},
\seqnum{A038207},
\seqnum{A055151},
\seqnum{A074909},
\seqnum{A101280}, and
\seqnum{A135278}.)


\begin{thebibliography}{99}

\bibitem{gamma} P. Barry, The $\gamma$-vectors of Pascal-like triangles defined by Riordan arrays, arXiv:1804.05027 [math.CO].

\bibitem{Book} P. Barry, \emph{Riordan Arrays: A Primer}, Logic Press, 2017.

\bibitem{CFT} P. Barry, Continued fractions and
transformations of integer sequences, \emph{J. Integer Seq.}, \textbf{12} (2009),
\href{https://cs.uwaterloo.ca/journals/JIS/VOL12/Barry3/barry93.pdf} {Article 09.7.6}.

\bibitem{Exp} P. Barry, On a family of generalized Pascal triangles defined by exponential Riordan array,
    \emph{J.
    Integer Seq.}, \textbf{10} (2007),
    \href{http://www.cs.uwaterloo.ca/journals/JIS/VOL10/Barry/barry202.html}{Article 07.3.5}.

\bibitem{Cons} P. Barry, On integer-sequence-based constructions of
generalized Pascal triangles, \emph{J. Integer Seq.}, \textbf{9} (2006),
\href{https://cs.uwaterloo.ca/journals/JIS/VOL9/Barry/barry91.pdf}{Article 06.2.4}.


\bibitem{Riordan_Exp} E. Deutsch and L. Shapiro, Exponential Riordan arrays, Lecture Notes, Nankai University, 2004,
available electronically at \newline
\href{http://www.combinatorics.net/ppt2004/Louis%20W.%20Shapiro/shapiro.htm}{
    http://www.combinatorics.net/ppt2004/Louis\%20W.\%20Shapiro/shapiro.htm}.

\bibitem{Fomin} S. Fomin and N. Reading, Root systems and
generalized associahedra, arXiv:0505518v3 [math.CO].

\bibitem{Petersen} K. Petersen, \emph{Eulerian Numbers}, Birkh\"auser, 2015.

\bibitem{Survey} L. Shapiro, A survey of the Riordan group, available electronically at
\href{http://www.combinatorics.cn/activities/Riordan\%20Group.pdf} {Center for Combinatorics}, Nankai University, 2018.

\bibitem{SGWW} L. W. Shapiro, S. Getu, W.-J. Woan, and L. C.
    Woodson,
The Riordan group, \emph{Discr. Appl. Math.} \textbf{34}
(1991), 229--239.


\bibitem{SL1} N. J. A.~Sloane, \emph{The
On-Line Encyclopedia of Integer Sequences}. Published electronically
at \href{http://oeis.org}{http://oeis.org}, 2018.

\bibitem{SL2} N. J. A.~Sloane, The On-Line Encyclopedia of Integer
Sequences, \emph{Notices Amer. Math. Soc.}, \textbf{50} (2003),  912--915.

\bibitem{Stanley} R. P. Stanley, $f$-vectors and $h$-vectors of simplicial posets, \emph{J. Pure Appl. Algebra}, \textbf{71} (1991), 319--331.

\bibitem{Wall} H.~S. Wall, \emph{Analytic Theory of
    Continued Fractions}, AMS Chelsea Publishing, 2001.

\bibitem{Wang} W. Wang and T. Wang, Generalized Riordan arrays,\emph{Discrete Math.}, \textbf{308} (2008), 6466--6500.
    
\bibitem{Ziegler} G. M. Ziegler,  \emph{Lectures on Polytopes}, Springer-Verlag, 1995.

\end{thebibliography}
\end{document}